\documentclass[12pt, a4paper]{amsart}
\pdfoutput=1
\usepackage{amsmath,amsfonts,amssymb,amsthm}  

\usepackage[utf8]{inputenc}

\theoremstyle{plain}
\newtheorem{theorem}{Theorem}[section]

\newtheorem{lemma}[theorem]{Lemma}
\newtheorem{proposition}[theorem]{Proposition}

\theoremstyle{definition}
\newtheorem{definition}[theorem]{Definition}

\newtheorem{remark}[theorem]{Remark}

\numberwithin{equation}{section}
\newtheorem*{theorem*}{Theorem}

\newcommand{\R}{{\mathbb R}}

\def\Xint#1{\mathchoice
{\XXint\displaystyle\textstyle{#1}}
{\XXint\textstyle\scriptstyle{#1}}
{\XXint\scriptstyle\scriptscriptstyle{#1}}
{\XXint\scriptscriptstyle\scriptscriptstyle{#1}}
\!\int}
\def\XXint#1#2#3{{\setbox0=\hbox{$#1{#2#3}{\int}$}
\vcenter{\hbox{$#2#3$}}\kern-.5\wd0}}

\def\dashint{\Xint-}

\newcommand{\dla}{\, \mathrm{d} \lambda}

\newcommand{\loc}{{\mbox{\scriptsize{loc}}}}

\DeclareMathOperator{\dive}{div}

\DeclareMathOperator{\BMO}{BMO}
\DeclareMathOperator{\PBMO}{PBMO}

\providecommand{\abs}[1]{ \lvert#1  \rvert}
\providecommand{\norm}[1]{ \lVert#1  \rVert}

\title[Parabolic weighted norm inequalities]{Parabolic weighted norm inequalities and partial differential equations}
\author{Juha Kinnunen and Olli Saari}
\address{Department of Mathematics, Aalto University, P.O. Box 11100, 
FI-00076 Aalto University, Finland}
\email{juha.k.kinnunen@aalto.fi, olli.saari@aalto.fi}

\thanks{The research is supported by the Academy of Finland and the V\"ais\"al\"a Foundation.}

\begin{document}
\begin{abstract}
We introduce a class of weights related to the regularity theory of nonlinear parabolic partial differential equations. 
In particular, we investigate connections of the parabolic Muckenhoupt weights to the parabolic $\BMO$.
The parabolic Muckenhoupt weights need not be doubling and they may grow arbitrarily fast in the time variable.
Our main result characterizes them through weak and strong type weighted norm inequalities for forward-in-time maximal operators. In addition, we prove a Jones type factorization result for the parabolic Muckenhoupt weights and a 
Coifman-Rochberg type characterization of the parabolic $\BMO$ through maximal functions.
Connections and applications to the doubly nonlinear parabolic PDE are also discussed.
\end{abstract}

\maketitle

\section{Introduction}

Muckenhoupt's seminal result characterizes weighted norm inequalities for the Hardy-Littlewood maximal operator through the so called $A_p$ condition
\[ \sup_{Q} \dashint_{Q} w \left( \dashint_{Q} w^{1-p'} \right)^{p-1} < \infty, \quad 1<p< \infty . \]
Here the supremum is taken over all cubes $Q \subset \mathbb{R}^{n}$,  and $w \in L^{1}_{loc}(\mathbb{R}^{n})$ is a nonnegative weight. These weights exhibit many properties that are powerful in applications, such as reverse H\"older inequalities, factorization property, and characterizability through $\BMO$, where $\BMO$ refers to the functions of bounded mean oscillation. 
Moreover, the Muckenhoupt weights play a significant role in the theory of Calder\'on-Zygmund singular integral operators, see \cite{GR1985}. 

Another important aspect of the Muckenhoupt weights and $\BMO$ is that they also arise in the regularity theory of nonlinear PDEs. 
More precisely, the logarithm of a nonnegative solution to any PDE of the type
\[
\dive (\abs{ \nabla u}^{p-2} \nabla u) =0, \quad1<p<\infty,
\]
belongs to $\BMO$ and the solution itself is a Muckenhoupt weight. This was the crucial observation in \cite{Moser1961}, where Moser proved the celebrated Harnack inequality for nonnegative solutions of such equations.

Even though the theory of the Muckenhoupt weights is well established by now, many questions related to higher dimensional versions of the one-sided Muckenhoupt condition
\[\sup_{x \in \mathbb{R}, h > 0} \frac{1}{h} \int_{x-h}^{x} w \left( \frac{1}{h} \int_{x}^{x+h} w^{1-p'} \right)^{p-1} < \infty  \]
remain open. This condition was introduced by Sawyer \cite{Sawyer1986} in connection with ergodic theory. Since then these weights and the one-sided maximal functions have been a subject of intense research; see \cite{AC1998}, \cite{AFMR1997}, \cite{CUNO1995}, \cite{Martin1993}, \cite{MOST1990}, \cite{MRPT1993}, \cite{MRT1992}, \cite{MRT1994} and \cite{Sawyer1986}. In comparison with the classical $A_p$ weights, the one-sided $A_p^{+}$ weights can be quite general. 
For example, they may grow exponentially, since any increasing function belongs to  $A_p^{+}$. 
It is remarkable that this class of weights still allows for weighted norm inequalities for some special classes of singular integral operators (see \cite{AFMR1997}), but the methods are limited to the dimension one.

The first extensions to the higher dimensions of the one-sided weights are by Ombrosi \cite{Ombrosi2005}. The subsequent research in \cite{Berkovits2011}, \cite{FMRO2011} and \cite{LO2010} contains many significant advances, but even in the plane many of the most important questions, such as getting the full characterization of the strong type weighted norm inequalities for the corresponding maximal functions, have not received satisfactory answers yet.

In this paper, we propose a new approach which enables us to solve many of the previously unreachable problems. In contrast with the earlier attempts, our point of view is related to Moser's work on the parabolic Harnack inequality in \cite{Moser1964} and  \cite{Moser1967}. 
More precisely, in the regularity theory for the doubly nonlinear parabolic PDEs of the type
\begin{equation}
\label{intro:equation}
\frac{\partial(|u|^{p-2}u)}{\partial t} - \dive (\abs{\nabla u}^{p-2} \nabla u)=0,  \quad1<p<\infty,
\end{equation}
(see \cite{GV2006}, \cite{KK2007}, \cite{KSU}, \cite{Trudinger1968}), there is a condition (Definition \ref{def:MAq}) that plays a role identical to that of the classical Muckenhoupt condition in the corresponding elliptic theory. 
Starting from the parabolic Muckenhoupt condition
\begin{equation}\label{intro:muck}
\sup_{R}  \dashint_{R ^{-}} w  \left( \dashint_{R  ^{+}} w^{1-q'}   \right)^{q-1}  < \infty,
\quad 1<q<\infty,
\end{equation}
where $R^{\pm}$ are space time rectangles with a time lag, we create a theory of parabolic weights. Here we use $q$ to distinguish from $p$ in the doubly nonlinear equation.
Indeed, they are not related to each other.

The time variable scales as the modulus of the space variable raised to the power $p$ in the geometry natural for \eqref{intro:equation}. Consequently, the Euclidean balls and cubes have to be replaced by parabolic rectangles respecting this scaling in all estimates. In order to generalize the one-sided theory of weighted norm inequalities, it would be sufficient to work with the case $p = 2$.
However, in view of the connections to nonlinear PDEs (see \cite{Saari2014} and \cite{KS2016}), we have decided to develop a general theory for $1<p<\infty$. As far as we know, the results in this work are new even for the heat equation with $p=2$. There are no previous studies about weighted norm inequalities with the same optimal relation to solutions of parabolic partial differential equations. 

Observe that the theory of parabolic weights contains the classical $A_p$ theory as a special case. However, the difference between elliptic and parabolic weights is not only a question of switching from cubes to parabolic rectangles.
There is an extra challenge in the regularity theory of \eqref{intro:equation} because of the time lag appearing in the estimates. 
A similar phenomenon also occurs in the harmonic analysis with one-sided weights, and it has been the main obstacle in the previous approaches  \cite{Berkovits2011}, \cite{FMRO2011}, \cite{LO2010}, and \cite{Ombrosi2005}.
Except for the one-dimensional case, an extra time lag appears in the arguments. Roughly speaking, a parabolic Muckenhoupt condition without a time lag implies boundedness of maximal operators with a time lag. In our approach, both the maximal operator and the Muckenhoupt condition have a time lag. This allows us to prove the necessity and sufficiency of the parabolic Muckenhoupt condition for both weak and strong type weighted norm inequalities of the corresponding maximal function. Our main technical tools are covering arguments related to the work of Ombrosi et al. \cite{Ombrosi2005}, \cite{FMRO2011};  parabolic chaining arguments from \cite{Saari2014}, and a Calder\'on-Zygmund argument based on a slicing technique.

Starting from the parabolic Muckenhoupt condition \eqref{intro:muck}, we build a complete parabolic theory of one-sided weighted norm inequalities and $\BMO$ in the multidimensional case. Our main results are a reverse H\"older inequality (Theorem \ref{thm:RHI}), strong type characterizations for weighted norm inequalities for a parabolic forward-in-time maximal function (Theorem \ref{thm:strongchar}), a Jones type factorization result for parabolic Muckenhoupt weights (Theorem \ref{thm:factorization}) and a Coifman-Rochberg type characterization of parabolic $\BMO$ through maximal functions (Theorem \ref{thm:coif-roch}). 
In Section 8, we explain in detail the connection between parabolic Muckenhoupt weights and the doubly nonlinear equation.
We refer to \cite{Aimar1988}, \cite{FG1985}, \cite{KK2007}, \cite{Moser1964}, \cite{Moser1967}, \cite{Saari2014} and \cite{Trudinger1968} for more on parabolic $\BMO$ and its applications to PDEs. 

\section{Notation}
Throughout the paper, the $n$ first coordinates of $\mathbb{R}^{n+1}$ will be called \textit{spatial} and the last one \textit{temporal}. The temporal translations will be important in what follows. Given a set $E \subset \mathbb{R}^{n+1}$ and $t \in \mathbb{R}$, we denote
\[E+t := \{e + (0,\ldots,0,t): e  \in E  \} .\]
The exponent $p$, with $1<p<\infty$, related to the doubly nonlinear equation \eqref{intro:equation} will be a fixed throughout the paper.

Constants $C$ without subscript will be generic and the dependencies will be clear from the context. We also write $K \lesssim 1$ for $K \leq C$ with $C$ as above. The dependencies can occasionally be indicated by subscripts or parentheses such as $K = K(n,p) \lesssim_{n,p} 1$.

A weight will always mean a real valued positive locally integrable function on $\mathbb{R}^{n+1}$. Any such function $w$ defines a measure absolutely continuous with respect to Lebesgue measure, and for any measurable $E \subset \mathbb{R}^{n+1}$, we denote 
\[w(E) := \int_{E} w .\]
We often omit mentioning that a set is assumed to be measurable. They are always assumed to be. 
For a locally integrable function $f$, the integral average is denoted as
\[\frac{1}{\abs{E}} \int_{E} f = \dashint_{E} f = f_{E}.\]
The positive part of a function $f$ is  $(f)^{+}=(f)_{+} = 1_{\{f > 0\}}f$ and the negative part  $(f)^{-}= (f)_{-} = -1_{\{f < 0\}}f$. 

\section{Parabolic Muckenhoupt weights}
Before the definition of the parabolic Muckenhoupt weights, we introduce the parabolic space-time rectangles in the natural geometry for the doubly nonlinear equation. 

\begin{definition}
\label{def:rectangles}
Let $Q(x,l) \subset \mathbb{R}^{n}$ be a cube with center $x$ and side length $l$ and sides parallel to the coordinate axes.
Let $p>1$ and $\gamma \in [0,1)$. We denote
\[
R(x,t,l )= Q(x,l) \times (t-  l^{p},t+ l^{p})
\]
and
\[
R ^{+}(\gamma) = Q(x,l) \times (t+ \gamma  l^{p},t+ l^{p}).
\]
The set $R(x,t,l)$ is called a $(x,t)$-centered parabolic rectangle with side $l$. We define $R^{-}(\gamma)$ as the reflection of $R ^{+}(\gamma)$ with respect to $\mathbb{R}^{n} \times \{t\}$. The shorthand $R^{\pm}$ will be used for $R^{\pm}(0)$.
\end{definition} 

Now we are ready for the definition of the parabolic Muckenhoupt classes. Observe that there is a time lag in the definition for $\gamma>0$.

\begin{definition}
\label{def:MAq}
Let $q > 1$ and $\gamma \in [0,1)$. A weight $w > 0$ belongs to the parabolic Muckenhoupt class $A_q^{+}(\gamma)$, if 
\begin{equation}
\label{eq:MA2}
\sup_{R} \left( \dashint_{R ^{-}(\gamma)} w   \right) \left( \dashint_{R  ^{+}(\gamma)} w^{1-q'}   \right)^{q-1} =: [w]_{A_q^{+}(\gamma)} < \infty.
\end{equation}
If the condition above is satisfied with the direction of the time axis reversed, we denote $w \in A_q^{-}(\gamma)$. 
If $\gamma$ is clear from the context or unimportant, it will be omitted in the notation.
\end{definition}

The case $A_2^{+}(\gamma)$ occurs in the regularity theory of parabolic equations, see  \cite{Moser1964} and \cite{Trudinger1968}. Before investigating the properties of  parabolic  weights, we briefly discuss how they differ from the ones already present in the literature. The weights of \cite{FMRO2011} and \cite{LO2010} were defined on the plane, and the sets $R^{\pm}(\gamma)$ in Definition \ref{def:MAq} were replaced by two squares that share exactly one corner point. The definition used in \cite{Berkovits2011} is precisely the same as our Definition \ref{def:MAq} with $p=1$ and $\gamma=0$. 
 
An elementary but useful property of the parabolic Muckenhoupt weights is that they can effectively be approximated by bounded weights. 

\begin{proposition}
\label{prop:truncation}
Assume that  $u,v \in A_q^{+}(\gamma)$. Then $f = \min \{u,v\} \in A_q^{+}(\gamma)$ and
\[[f]_{A_q^{+}} \lesssim [u]_{A_q^{+}}  + [v]_{A_q^{+}} . \]
The corresponding result holds for $\max \{u,v\}$ as well.
\end{proposition}
\begin{proof}
A direct computation gives
\begin{align*}
 \left( \dashint_{R^{-}(\gamma)} f \right) &\left( \dashint_{R^{+}(\gamma)} f^{1-q'} \right)^{q-1} \\
& \lesssim  \left( \dashint_{R^{-}(\gamma)} f \right) \left( \frac{1}{\abs{R^{+}(\gamma)}} \int_{R^{+}(\gamma) \cap \{u > v\}}  f^{1-q'} \right)^{q-1} \\
& \hspace{2.2cm} + \left( \dashint_{R^{-}(\gamma)} f \right)\left(\frac{1}{\abs{R^{+}(\gamma)}} \int_{R^{+}(\gamma)\cap \{u \leq v\}} f ^{1-q'} \right)^{q-1} \\
&\leq  \left( \dashint_{R^{-}(\gamma)} v \right) \left( \frac{1}{\abs{R^{+}(\gamma)}} \int_{R^{+}(\gamma) \cap \{u > v\}}  v^{1-q'} \right)^{q-1} \\
& \hspace{2.2cm} + \left( \dashint_{R^{-}(\gamma)} u \right)\left( \frac{1}{\abs{R^{+}(\gamma)}} \int_{R^{+}(\gamma)\cap \{u \leq v\}} u ^{1-q'} \right)^{q-1} \\
& \leq [u]_{A_q^{+}}  + [v]_{A_q^{+}}.
\end{align*}
The result for $\max\{u,v\}$ is proved in a similar manner.
\end{proof}

\subsection{Properties of parabolic Muckenhoupt weights}
The special role of the time variable makes the parabolic Muckenhoupt weights quite different from the classical ones. For example, the doubling property does not hold, but it can be replaced by a weaker forward-in-time comparison condition. The next proposition is a collection of useful facts about the parabolic Muckenhoupt condition, the most important of which is the property that the value of $\gamma \in [0,1)$ does not play as big a role as one might guess. This is crucial in our arguments. The same phenomenon occurs later in connection with the parabolic $\BMO$. 

\begin{proposition}
\label{prop:properties}
Let $\gamma \in [0,1)$. Then the  following properties hold true.
\begin{enumerate}
\item[(i)] If $1< q < r < \infty$, then $A_q^{+}(\gamma)  \subset A_r^{+}(\gamma) $.
\item[(ii)] Let $\sigma = w^{1-q'}$. Then $\sigma $ is in $A_{q'}^{-}(\gamma)$ if and only if $w \in A_q^{+}(\gamma)$.
\item[(iii)] Let $w \in A_q^{+}(\gamma)$, $\sigma =w^{1-q'}$ and $t > 0$. Then
\[\dashint_{R^{-}(\gamma )} w \leq C_t \dashint_{t + R^{-}(\gamma )} w \quad \textrm{and} \quad \dashint_{R^{+}(\gamma )} \sigma \leq C_t \dashint_{-t + R^{+}(\gamma)} \sigma.\]  
\item[(iv)] If $w \in A_q^{+}(\gamma)$, then we may replace $R^{-}(\gamma)$ by $R^{-}(\gamma) - a$ and  $R^{+}(\gamma)$ by $R^{+}(\gamma) + b$ for any $a,b \geq 0$ in the definition of the parabolic Muckenhoupt class. The new condition is satisfied with a different constant $[w]_{A_q^{+}}$.
\item[(v)] If $1>\gamma' > \gamma$, then $A_q^{+}(\gamma) \subset A_q^{+}(\gamma')$.
\item[(vi)] Let $w \in A_q^{+}(\gamma)$. Then
\[w(R^{-}(\gamma)) \leq C \left(\frac{\abs{R^{-}(\gamma)}}{\abs{S}}  \right)^{q} w(S) \]
 for every  $S \subset R^{+}(\gamma)$.  
\item[(vii)] If $w \in A_q^{+} (\gamma)$ with some $\gamma \in [0,1)$, then $w\in A_q^{+}(\gamma')$ for all $\gamma' \in (0,1) $. 
\end{enumerate}
\end{proposition}
\begin{proof}
First we observe that (i) follows from H\"older's inequality and (ii) is obvious.  For the case $t + R^{-}(\gamma) = R^{+}(\gamma)$ the claim (iii) follows from Jensen's inequality. For a general $t$, the result  follows from subdividing the rectangles $R^{\pm}(\gamma)$ into smaller and possibly overlapping subrectangles and applying the result to them. The property (iv) follows directly from (iii), as does (v) from (iv). 

For (vi), take $S \subset R^{+}(\gamma)$ and let $f = 1_{S}$. Apply the $A_q^{+}(\gamma)$ condition to see that
\begin{align*}
\left(\frac{\abs{S}}{\abs{R^{+}(\gamma)}}  \right)^{q} w(R^{-}(\gamma)) 
&= (f_{R^{+}(\gamma)})^{q}  w(R^{-}(\gamma)) \\
& \leq  \left( \dashint_{R^{+}(\gamma)} f^{q} w \right)  \left(\dashint_{R^{+}(\gamma)} w^{1-q'} \right)^{q/q'} w(R^{-}(\gamma)) \\
& \leq C   w(S).
\end{align*}

For the last property (vii), take $R = Q(x,l) \times (t- l^{p},t+l^{p})$. Let $\gamma \in (0,1)$ and suppose that $w \in A_q^{+}(\gamma)$. We will prove that the condition $A_q^{+}(2^{-1} \gamma)$ is satisfied. We subdivide $Q$ into $2^{nk}$ dyadic subcubes $\{Q_i\}_{i=1}^{2^{nk}}$. This gives dimensions for the lower halves of parabolic rectangles $R_i^{-}(\gamma)$. For a given $Q_i$, we stack a minimal amount of the rectangles $R_i^{-}(\gamma)$ so that they almost pairwise disjointly cover $Q_i \times (t- l^{p}, t- 2^{-1}\gamma l^{p})$. The number of $R_i^{-}(\gamma)$ needed to cover $Q  \times (t- l^{p}, t- 2^{-1}\gamma l^{p})$ is bounded by
\[2^{nk} \cdot \frac{(1-2^{-1}\gamma)l^{p}}{2^{-nkp}(1- \gamma)l^{p}} = 2^{nk(p+1)} \frac{2-\gamma}{2(1-\gamma)}.\]

Corresponding to each $Q_i$, there is a sequence of at most $2^{k}-1$ vectors $d_j = 2^{- k-1}l e_j$ with $e_j  \in \{0,1\}^{n}$ such that 
\[Q_i + \sum_{j} d_j = 2^{- k} Q. \]
Next we show how every rectangle $R_{i}(\gamma)$ can be transported to the same spatially central position $2^{- k} Q$ without losing too much information about their measures. By (vi) we have
\[w(R_{i}^{-} (\gamma)) \leq C \left(\frac{\abs{R_i^{-}(\gamma)}}{\abs{S}} \right)^{q} w(S) \] 
for any $S \subset R_i^{+}(\gamma)$. We choose $S$ such that its projection onto space variables is $(Q_i + d_1) \cap Q_i $, and its projection onto time variable has full length $(1-\gamma)(2^{- k}l)^{p}$. Then
\[w(R_{i}^{-} (\gamma)) \leq C_0 w(S) \leq C_0 w(R_{i}^{1-}(\gamma)) \]
where $R_{i}^{1-}(\gamma) \supset S$ is $Q_i+ d_1$ spatially and coincides with $S$ as a temporal projection. The constant $C_0$ depends on $n$ and $q$. 

Next we repeat the argument to obtain a similar estimate for $R_{i}^{1-}(\gamma)$ in the place of $R_{i}^{-} (\gamma)$. We obtain a new rectangle to the right hand side, on which we repeat the argument again. With $k$ iterations, we reach an inequality
\[w(R_{i}^{-} (\gamma)) \leq C_0^{2^{k}-1} w(R_{i}^{*-} (\gamma)) \]
where $R_{i}^{*-} (\gamma) $ is the parabolic box whose projection onto the coordinates corresponding to the space variables is $2^{-k}Q$. The infimum of time coordinates of points in $R_{i}^{*-} (\gamma) $ equals 
\[ \inf \{t: (x,t) \in R_{i}^{-} \} + (2^{k} - 1)(1+\gamma)(2^{- k} l)^{p}. \] 
As $p > 1$, the second term in this sum can be made arbitrarily small. In particular, for a large enough $k$, we have 
\[(2^{k} - 1)(1+\gamma)(2^{- k} l)^{p} \leq 2\cdot 2^{-k(p-1)} l^{p} \leq  \frac{1}{100} \gamma l^{p}. \]

In this fashion, we may choose a suitable finite $k$ and divide the sets $R^{\pm}(2^{-1}\gamma)$ to
$N \lesssim_{n,\gamma} 2^{nkp}$
parts $R_{i}^{\pm}(\gamma)$. They satisfy
\[
w(R_{i}^{-}(\gamma)) \leq C_0^{2^{k}-1} w(R_{i}^{*-}(\gamma))
\]
and 
\[
\sigma(R_{i}^{+}(\gamma)) \leq C_0^{2^{k}-1} \sigma(R_{i}^{*+}(\gamma)),
\]
where all starred rectangles have their projections onto space variables centered at $2^{- k} Q$; they have equal side length $2^{- kp}l^{p}$, and 
\[ \frac{1}{2}\gamma l^{p} \leq d(R_{i}^{*-}(\gamma),R_{j}^{*+}(\gamma)) < 2 l^{p} \]
for all $i, j$. All this can be done by a choice of $k$ which is uniform for all rectangles.
 
It follows that 
\begin{align*}
&\left( \dashint_{R ^{-}(2^{-1}\gamma)} w   \right) \left( \dashint_{R  ^{+}(2^{-1}\gamma)} w^{1-q'}   \right)^{q-1} \\
& \lesssim \sum_{i,j = 1}^{N} \left( \dashint_{R_i ^{-}(\gamma)} w   \right) \left( \dashint_{R_j  ^{+}(\gamma)} w^{1-q'}   \right)^{q-1} \\
& \lesssim \sum_{i,j = 1}^{N} \left( \dashint_{R_i ^{*-}(\gamma)} w   \right) \left( \dashint_{R_j  ^{*+}(\gamma)} w^{1-q'}   \right)^{q-1} \\
& \lesssim   \sum_{i,j = 1}^{N} C 
= C(n,p,k,\gamma,q,[w]_{A_q^{+}}(\gamma)),
\end{align*}
where in the last inequality we used (iv). Since the estimate is uniform in $R$, the claim follows.
\end{proof}


\section{Parabolic maximal operators}
In this section, we will study parabolic forward-in-time maximal operators, which are closely related to the one-sided maximal operators studied in \cite{Berkovits2011}, \cite{FMRO2011} and \cite{LO2010}. The class of weights in \cite{FMRO2011}, originally introduced by Ombrosi \cite{Ombrosi2005}, characterizes the weak type inequality for the corresponding maximal operator, but the question about the strong type inequality remains open. On the other hand, Lerner and Ombrosi \cite{LO2010} managed to show that the same class of weights supports strong type boundedness for another class of operators with a time lag. For the boundedness of these operators, however, the  condition on weights  is not necessary. Later the techniques developed by Berkovits \cite{Berkovits2011} showed that a weight condition without a time lag implies boundedness of maximal operators with a time lag. That approach applied to all dimensions. In our case both the maximal operator and the Muckenhoupt condition have a time lag. This approach, together with scaling of parabolic rectangles, allows us to prove both the necessity and sufficiency of the parabolic Muckenhoupt condition for weak and strong type weighted norm inequalities for the maximal function to be defined next.

\begin{definition}
Let $\gamma \in [0,1)$. For $f \in L_{loc}^{1}(\mathbb{R}^{n+1})$ define the parabolic maximal function
\[M^{\gamma  +}f(x,t) = \sup _{R(x,t)} \dashint_{R^{+}(\gamma)} \abs{f}  , \]
where the supremum is taken over all parabolic rectangles centered at $(x,t)$. If $\gamma = 0$, it will be omitted in the notation. The operator $M^{\gamma -}$ is defined analogously.
\end{definition}


The necessity of the $A_q^{+}$ condition can be proved in a similar manner to its analogue in the classical Muckenhoupt theory, but already here the geometric flexibility of Definition \ref{def:MAq} simplifies the statement.

\begin{lemma}
\label{lemma:weaktypeap}
Let $w$ be a weight such that the operator $M^{\gamma+} : L^{q}(w) \rightarrow L^{q,\infty}(w)$ is bounded. Then $w \in A_q^{+}(\gamma)$. 
\end{lemma} 
\begin{proof}
Take $f > 0$ and choose $R$ such that $f_{S^{+}} > 0$, where $S^{+} = R^{+} $ if $\gamma =0$. If $\gamma  > 0$, 
\[S^{+} = R^{-}(\gamma) + (1-\gamma)l^{p} + 2^{p} \gamma l^{p}  \] 
will do. Redefine $f = \chi_{S^{+}}f$. Take a positive $\lambda < C_{\gamma} f_{S^{+}}$ . With a suitably chosen $C_\gamma$, we have
\begin{align*}
w(R^{-}) \leq w(\{x\in \mathbb{R}^{n+1}: M^{\gamma+}f > \lambda \}) \leq \frac{C}{\lambda^{q}} \int_{R^{+}} f^{q}w.
\end{align*} 
The claim follows letting $\lambda \to C_\gamma f= C_\gamma(w +  \epsilon)^{1-q'}$ and $\epsilon \to 0$, and concluding by argumentation similar to Proposition \ref{prop:properties}.
\end{proof}

\subsection{Covering lemmas}
The converse claim requires a couple of special covering lemmas. It is not clear whether the main covering lemma in \cite{FMRO2011} extends to dimensions higher than two. However, in our geometry the halves of parabolic rectangles are indexed along their spatial centers instead of corner points, which was the case in \cite{FMRO2011}. This fact will be crucial in the proof of Lemma \ref{lemma:cover2}, and this enables us to obtain results in the multidimensional case as well.

\begin{lemma}
\label{lemma:cover1}
Let $R_0$ be a parabolic rectangle, and let $\mathcal{F}$ be a countable collection of parabolic rectangles with dyadic sidelengths such that for each $i \in \mathbb{Z}$ we have
\[\sum_{\substack{ P  \in \mathcal{F}\\ l(P) = 2^{i} }} 1_{P^{-} } \lesssim 1. \]
Moreover, assume that $P^{-} \nsubseteq R^{-}$ for all distinct $P,R \in \mathcal{F}$. Then
\[\sum_{P \in \mathcal{G}} \abs{P} \lesssim \abs{R_0}, \]
where $\mathcal{G} = \{P \in \mathcal{F}: P^{+} \cap R_0^{+} \neq  \varnothing, \abs{P} < \abs{R_0} \}$.  
\end{lemma}

\begin{proof}
Recall that $R^{\pm} = R^{\pm}(0)$. 
We may write $\mathcal{G} \subset \mathcal{G}_0(R_0) \cup \mathcal{G}_1$, where
\[
\mathcal{G}_0(R ) = \{P \in \mathcal{F}: P \cap \partial R ^{+} ,\abs{P} < \abs{R } \}
\]
and
\[
\mathcal{G}_1 = \{P \in \mathcal{F}: P \subset   R_0^{+} ,\abs{P} < \abs{R_0} \}.
\]
That is, the rectangles having their upper halves in $R_0^{+}$ are either contained in it or they meet its boundary. An estimate for $\mathcal{G}_0(R)$ with an arbitrary parabolic rectangle $R$ instead of $R_0$ will be needed, so that we start with it. Let $P$ be a parabolic rectangle with the spatial side length $l(P) = 2^{-i}$. If $P \cap \partial R^{+} \neq \varnothing$, then $P \subset A_i $, where $A_i$ can be realized as a collection of $2(n+1)$ rectangles corresponding to each face of $R$ such that
\begin{align*}
\abs{A_i} \lesssim 2 l(R)^{n} \cdot 2^{-ip} + 2n l(R)^{p+n-1} \cdot 2^{-i}.
\end{align*}
Now choosing $k_0 \in \mathbb{Z}$ such that $2^{-k_0} < l(R) < 2^{-k_0 + 1}$, we get, by the bounded overlap, 
\begin{align*}
\sum_{P \in \mathcal{G}_0(R)} \abs{P} = \sum_{i=k_0}^{\infty} \sum_{\substack{P \in \mathcal{G}_0(R)\\ l(P) = 2^{-i} }} \abs{P} \lesssim \sum_{i=k_0}^{\infty}  \abs{A_i} \lesssim \abs{R}.
\end{align*}

Once the rectangles meeting the boundary are clear, we proceed to $\mathcal{G}_1$. The side lengths of rectangles in $\mathcal{G}_1$ are bounded from above. Hence there is at least one rectangle with the maximal side length. Let $\Sigma_1$ be the collection of $R \in \mathcal{G}_1$ with the maximal side length. We continue recursively. Once $\Sigma_j$ with $j=1,\ldots,k$ have been chosen, take the rectangles $R$ with the maximal side length among the rectangles in $\mathcal{G}_1$ satisfying
\[R^{-} \cap \bigcup_{P \in \cup_{j=1}^{k} \Sigma_j} P^{-} = \varnothing.\] 
Let them form the collection $\Sigma_{k+1}$. Define the limit collection to be
 \[\Sigma = \bigcup_{j} \Sigma_j. \]

Each $P \in \mathcal{G}_1$ is either in $\Sigma$ or $P^{-}$ meets $R^{-}$ with $R \in \Sigma$ and $l(P) < l(R)$. Otherwise $P$ would have been chosen to $\Sigma$. This implies that
\[\sum_{R \in \mathcal{G}_1} \abs{R} \leq \sum_{R \in \mathcal{G}_1 \cap \Sigma }\left( \abs{R} +\sum_{\substack{P \in \mathcal{G}_1: P^{-} \cap R^{-} \neq \varnothing \\
\abs{P} < \abs{R}}} \abs{P} \right)  . \] 
In the second sum, both $P$ and $R$ are in $\mathcal{F}$, so $P^{-} \nsubseteq R^{-}$ by assumption. Thus $P \cap \partial R^{-}  \neq \varnothing$, and the sum in the parentheses is controlled by a constant multiple of $\abs{R}$ (by applying the estimate we have for $\mathcal{G}_0(\widetilde{R})$ where $\widetilde{R}$ is a parabolic rectangle with upper half $R^{-}$). The rectangles in each $\Sigma_j$ have equal side length so that
\begin{align*}
\sum_{R \in \mathcal{G}_1} \abs{R} & \lesssim \sum_{R \in \mathcal{G}_1 \cap \Sigma}  \abs{R} = \sum_{j } \sum_{R \in \mathcal{G}_1 \cap \Sigma_j }  \abs{R} 
\\
&  \lesssim \sum_{j } \Bigg \lvert \bigcup_{R \in \Sigma_j} R \Bigg \rvert  \leq  \Bigg \lvert \bigcup_{R \in \mathcal{G}_1}R \Bigg \rvert \leq \abs{R_0}.
\end{align*}
\end{proof}

The hypothesis of the next lemma correspond to a covering obtained using the parabolic maximal function, and the conclusion provides us with a covering that has bounded overlap. This fact is analogous to the two-dimensional Lemma 3.1 in \cite{FMRO2011}.

\begin{lemma}
\label{lemma:cover2}
Let $\lambda > 0$, $f \in L_{loc}^{1}(\mathbb{R}^{n+1})$ be nonnegative, and $A \subset \mathbb{R}^{n+1}$ a set of finitely many points such that for each $x \in A$ there is a parabolic rectangle $R_x$ with dyadic side length satisfying
\begin{equation}
\label{eq:cover20}
\dashint_{R_x^{+}(\gamma)} f \eqsim \lambda.  
\end{equation}
Then there is $\Gamma \subset A$ such that for each $x \in \Gamma$ there is $F_x \subset R_x ^{+}(\gamma)$ with
\begin{enumerate}
\item[(i)] $A \subset \bigcup_{x \in \Gamma}  \overline{R_x ^{-}}$,
\item[(ii)]
\[\frac{1}{\abs{R_x  }} \int_{F_x} f\gtrsim \lambda \quad \textrm{and} \quad \sum_{x \in \Gamma} 1_{F_x} \lesssim  1.\]
\end{enumerate}
\end{lemma}

\begin{proof}
To simplify the notation, we identify the sets $R_x^{-}$ with their closures. Their side lengths are denoted by $l_x$. Let $x_1 \in A$ be a point with maximal temporal coordinate. Recursively, choose $x_{k+1} \in A \setminus \bigcup_{j=1}^{k} R_x^{-} $. Denote $\Delta = \{x_i\}_i$. This is a finite set. Take $x \in \Delta$ with maximal $l_x$ and define $\Gamma_1 = \{x\}$. Let $\Gamma_{k+1} = \Gamma_k \cup \{y\}$ where $R_y ^{-} \nsubseteq R_x ^{-}$ for all $x \in \Gamma_k$ and $l_y$ is maximal among the $l_y$ satisfying the criterion. By finiteness the process will stop and let $\Gamma$  be the final collection. 

Given $x,y \in \Gamma$ with $ l_x = l_y =: r$ and $x \neq y$, their Euclidean distance satisfies 
\[\abs{x-y} \geq \min\left\{ \frac{1}{2} r, r^{p} \right \}. \]
There is a dimensional constant $\alpha \in (0,1)$ such that 
$\alpha R_x   \cap \alpha R_y  = \varnothing$, 
and, given $z \in \mathbb{R}^{n+1}$, there is a dimensional constant $\beta > 0$ such that 
\[\bigcup_{x \in \Gamma : z \in R_x } R_x   \subset R(z, \beta r). \]
Thus
\[
 (\beta r)^{n}(2\beta r )^{p} = \abs{R(z, \beta r)} \geq \sum_{\substack{x \in \Gamma: l_x = r, \\ z \in R_x  }} \abs{\alpha R_x } 
  = (\alpha r)^{n} (2\alpha r)^{p} \sum_{\substack{x \in \Gamma\\ l_x = r }} 1_{R_x }(z),     
\]
and consequently
\begin{equation}
\label{eq:cover21}
\sum_{\substack{x \in \Gamma\\ l_x = r }} 1_{ R_x } \lesssim 1.
\end{equation}

Denote 
\[
\mathcal{G}_x = \{y \in \Gamma: R_x^{+}(\gamma) \cap R_y^{+}(\gamma) \neq \varnothing , \abs{R_y} < \abs{R_x}  \}.
\] 
By inequality \eqref{eq:cover21} the assumptions of Lemma \ref{lemma:cover1} are fulfilled. Hence
\begin{equation*}
\sum_{y \in \mathcal{G}_x} \abs{R^{+}_y(\gamma)} \lesssim \abs{R^{+}_x(\gamma)}.
\end{equation*}
By \eqref{eq:cover20}, we have
\[\sum_{y \in \mathcal{G}_x} \int _{R_y ^{+}(\gamma)} f \lesssim \lambda \sum_{y \in \mathcal{G}_x} \abs{R_y^{+}(\gamma)} \lesssim \lambda \abs{R_x^{+}(\gamma)} \lesssim \int_{R_x^{+}(\gamma)} f.  \]
Let the constant in this inequality be $N$.

Denote $s := \# \mathcal{G}_x$. In case $s \leq 2N$, we choose $F_x = R_x^{+}(\gamma)$. If $s > 2N$, we define
\[E_i^{x} = \left\lbrace z \in R_x^{+}: \sum_{y \in \Gamma: l_y < l_x} 1_{R_y^{+}(\gamma)}(z) \geq i \right\rbrace.   \]
Thus $\sum_{i} 1_{E_i^{x}}(z)$ counts the points $y \in \mathcal{G}_x$ whose rectangles contain $z$. Hence
\begin{align*}
2N \int_{E_{2N}^{x}} f &\leq \sum_{i=1}^{s}  \int_{E_{i}^{x}} f = \int_{R_x^{+}(\gamma)} f  \sum_{i=1}^{s}  1_{E_{i}^{x}}   \\
&\leq  \int_{R_x^{+}(\gamma)} f  \sum_{y  \in \mathcal{G}_x}   1_{R_{y}^{+}(\gamma)} = \sum_{y  \in \mathcal{G}_x}  \int_{R_y^{+}(\gamma)} f \leq N \int_{R_x^{+}(\gamma)} f.
\end{align*}
For the set $F_x = R_x^{+}(\gamma) \setminus E_{2N}^{x}$ we have
\[\int_{F_x} f = \int_{R_x^{+}(\gamma)}f - \int_{E_{2N}^{x}}f \geq \frac{1}{2}  \int_{R_x^{+}(\gamma)}f \gtrsim \lambda \abs{R_x^{+}(\gamma)}. \]

It remains to prove the bounded overlap of $F_x$. Take $z \in \bigcap_{i=1}^{k}F_{x_i}$. Take $x_{j}$ so that $l_{x_{j}}$ is maximal among $l_{x_i}$, $i=1,\ldots,k$. By \eqref{eq:cover21} there are at most $C_n$ rectangles with this maximal side length that contain $z$. Moreover, their subsets $F_x$ meet at most $2N$ upper halves of smaller rectangles, so that $k \leq 2NC_n$.
\end{proof}

\subsection{Weak type inequalities}
Now we can proceed to the proof of the weak type inequality. The proof makes use of a covering argument as in \cite{FMRO2011} adjusted to the present setting.

\begin{lemma}
\label{lemma:apweaktype}
Let $q \geq 1$, $w \in A_q^{+}(\gamma)$ and $f \in L^{q}(w)$.  There is a constant $C= C(n,\gamma,p,w,q)$ such that
\[w(\{x \in \mathbb{R}^{n+1}: M^{\gamma+} f > \lambda \}) \leq \frac{C }{\lambda^{p}} \int \abs{f}^{p} w \]
for every $\lambda > 0$.
\end{lemma}
\begin{proof}
We first assume that $f > 0$ is bounded and compactly supported. Since
\begin{align*}
M^{\gamma+} f(x) &= \sup _{h > 0} \frac{1}{R(x,h, \gamma)^{+}} \int_{R(x,h, \gamma)^{+}} f \\
&\lesssim \sup _{i \in \mathbb{Z}} \frac{1}{R(x,2^{i},2^{-2} \gamma)^{+}} \int_{R(x,2^{i}, 2^{-2}\gamma )^{+}} f \\
&= \lim_{j \to -\infty} \sup_{i \in \mathbb{Z}; i> j} \frac{1}{R(x,2^{i}, \gamma')^{+}} \int_{R(x,2^{i}, \gamma')^{+}} f ,
\end{align*}
it suffices to consider rectangles with dyadic sidelengths bounded from below provided that we use smaller $\gamma$, and the claim will follow from monotone convergence. The actual value of $\gamma$ is not important because of Proposition \ref{prop:properties}. We may assume that $w$ is bounded from above and from below (see Proposition \ref{prop:truncation}). 

Moreover, it suffices to estimate $w(E)$, where 
\[
E= \{x \in \mathbb{R}^{n+1}: \lambda < M^{\gamma+} f \leq 2\lambda \}.
\] 
Once this has been done, we may sum up the estimates to get
\begin{align*}
w(\mathbb{R}^{n+1} \cap \{M^{ \gamma+} f > \lambda\}) 
&= \sum_{i=0}^{\infty} w(\mathbb{R}^{n+1} \cap \{2^{i}\lambda< M^{ \gamma+} f \le 2^{i+1}\lambda\}) 
\\
&\leq  \sum_{i=0}^{ \infty}  \frac{1}{2^{i}}\frac{C}{\lambda^{p}} \int \abs{f}^{p}w \leq \frac{C}{\lambda^{p}} \int \abs{f}^{p}w .
\end{align*}

Let $K \subset E$ be an arbitrary compact subset. Denote the lower bound for the sidelengths of the parabolic rectangles in the basis of the maximal operator by $\xi < 1$. For each $x \in K$ there is dyadic $l_x > \xi$ such that 
\[\dashint_{R^{+}(x,l_x, \gamma)} f \eqsim \lambda. \]
Denote $R_x := R(x,l_x)$. Since $f \in L^{1}$, we have
\[\abs{R_x^{+}(\gamma)} < \frac{1}{\lambda} \int f = C(\lambda, \norm{f}_{L^{1}}) < \infty.\] 
Thus $\sup_{x \in K} l_x < \infty $. Let $a = \min w$. There is $\epsilon  > 0$, uniform in $x$, such that 
\[w((1+\epsilon ) R_x^{-} \setminus R_x^{-} ) \leq a \xi^{n+p}   \leq w(R_x^{-})   \]
and $w((1+\epsilon) R_x^{-} ) \leq 2 w(  R_x^{-} ) $ hold for all $x \in K$. By compactness there is a finite collection of balls $B(x,\xi^{p} \epsilon / 2)$ to cover $K$. Denote the set of centers by $A$, apply Lemma \ref{lemma:cover2} to extract the subcollection $\Gamma$. Each $y \in K$ is in $B(x,\xi^{p} \epsilon / 2)$ with $x \in A$. Each $x \in A$ is in $R_z^{-}$ with $z \in \Gamma$, so each $y \in K$ is in $B(x,\xi^{p} \epsilon / 2) \subset (1+\epsilon)R_z^{-}$. Thus
\begin{align*}
w(K) &\leq \sum_{z \in \Gamma} w((1+\epsilon) R_z^{-}) \leq 2\sum_{z \in \Gamma} w(  R_z^{-}) \\
&\leq \frac{C}{\lambda^{q}} \sum_{z \in \Gamma} w(  R_z^{-}) \left( \frac{1}{\abs{R_z^{+}(\gamma)}} \int_{F_z} f \right)^{q} \\
&\leq \frac{C}{\lambda^{q}} \sum_{z \in \Gamma} \frac{ w(  R_z^{-})}{\abs{R_z^{-}}} \left(\dashint_{R_z^{+}(\gamma)} w^{1-q'} \right)^{q-1} \int_{F_z} f^{q} w \\
&\leq \frac{C}{\lambda^{q}} \int f^{q} w.
\end{align*}
In the last inequality we used the $A_q^{+}$ condition together with a modified configuration justified in Proposition \ref{prop:properties}, and the bounded overlap of the sets $F_z$.
\end{proof}

Now we are in a position to summarize the first results about the parabolic Muckenhoupt weights. We begin with the weak type characterization for the operator studied in \cite{Berkovits2011}. Along with this result, the definition in \cite{Berkovits2011} leads to all same results in $\mathbb{R}^{n+1}$ as the other definition from \cite{FMRO2011} does in $\mathbb{R}^{2}$. The next theorem holds even in the case $p =1$, which is otherwise excluded in this paper.  

\begin{theorem}
Let $w$ be a weight and $q > 1$. Then $w \in A_q^{+}(\gamma)$ with $\gamma = 0$ if and only if $M^{+}$ is of $w$-weighted weak type $(q,q)$.  
\end{theorem}
\begin{proof}
Combine Lemma \ref{lemma:weaktypeap} and Lemma \ref{lemma:apweaktype}.
\end{proof}

The next theorem is the first main result of this paper. Observe that all the parabolic operators $M^{\gamma+}$ with $\gamma  \in (0,1) $ have the same class of good weights. This interesting phenomenon seems to be related to the fact that $p> 1$.

\begin{theorem}
\label{thm:weakchar}
Let $w$ be a weight and $q > 1$. Then the following conditions are equivalent:
\begin{itemize}
\item[(i)] $w \in A_q^{+}$ for some $\gamma \in (0,1)$,
\item[(ii)] $w \in A_q^{+}$ for all $\gamma \in (0,1)$,
\item[(iii)] there is $\gamma \in (0,1)$ such that the operator $M^{\gamma+}$ is of weighted weak type $(q,q)$ with the weight $w$,
\item[(iv)] the operator $M^{\gamma+}$ is of weighted weak type $(q,q)$ with the weight $w$ for all $\gamma \in (0,1)$.
\end{itemize}
\end{theorem}
\begin{proof}
Lemma \ref{lemma:weaktypeap}, Lemma \ref{lemma:apweaktype} and (vii) of Proposition \ref{prop:properties}.
\end{proof}

\section{Reverse H\"older inequalities}
Parabolic reverse H\"older inequalities had already been studied in \cite{Berkovits2011}, and they were used to prove  sufficiency of the nonlagged Muckenhoupt condition for the lagged strong type inequality. The proof included the classical argument with self-improving properties and interpolation. Our reverse H\"older inequality will lead to an even stronger self-improving property, and this will give us a characterization of the strong type inequality. We will encounter several challenges. For example, our ambient space does not have the usual dyadic structure. In the classical Muckenhoupt theory this would not be a problem, but here the forwarding in time gives new complications. We will first prove an estimate for the level sets, and then we will use it to conclude the reverse H\"older inequality.  

\begin{lemma}
\label{lemma:quasistein}
Let $w \in A_q^{+}(\gamma)$, $\widetilde{R}_0 = Q_0 \times (\tau,\tau+ \frac{3}{2}l_0^{p})$ and $\widehat{R}_0 = Q_0 \times (\tau,\tau +  l_0^{p})$. 
Then there exist $C = C([w]_{A_q^{+}(\gamma)},n,p)$ and $\beta \in (0,1)$ such that for every $\lambda \geq w_{R_0^{-}}$, we have
\[w(\widehat{R}_0 \cap \{w > \lambda\} ) \leq C \lambda \abs{\widetilde{R}_0 \cap \{w > \beta \lambda\}}. \]

\end{lemma}
\begin{proof}
We introduce some notation first. For a parabolic rectangle $R = Q \times (t_0,t_0 + 2 l(Q)^{p})$, we define
\begin{align}
\label{def:hattu}
\widehat{R} &= Q \times (t_0,t_0 + l(Q)^{p}) \quad \textrm{and} \\
\label{def:mato}
\check{R} &= Q \times (t_0 +(1+\gamma) l(Q)^{p}, \textstyle{\frac{3}{2}} l(Q)^{p}   ).
\end{align} 
Here $\gamma \in (0,1/2)$, and by Proposition \ref{prop:properties}, we may replace the sets $R^{\pm}(\gamma)$ everywhere by the sets $\widehat{R}$ and $\check{R}$. Note that $\widehat{R} = R^{-}$. The hat is used to emphasize that $\widehat{R} $ and $\check{R}$ are admissible in the $A_q^{+}$ condition, whereas $R^{-}$ is used as the set should be interpreted as a part of a parabolic rectangle. For $\beta \in (0,1)$, the condition $A_q^{+}(\gamma)$ gives
\[
\abs{ \check{R} \cap \{w \leq \beta w_{\widehat{R}}\}}  \leq \beta^{p'-1} \int_{\check{R}} \frac{ w ^{1-p'}}{w_{\widehat{R}}^{1-p'}} 
\leq (\beta C)^{p'-1} \abs{\check{R}}.
\]
Taking $\alpha \in (0,1)$, we may choose $\beta$ such that 
\begin{equation}
\label{eq:quasi0}
\abs{\check{R} \cap \{w > \beta w_{\widehat{R}}\}} > \alpha \abs{\check{R}}.
\end{equation}

Let 
\[\mathcal{B} = \{Q \times (t-\textstyle{\frac{ 1}{2}} l(Q)^{p},t+\textstyle{\frac{ 1}{2} }l(Q)^{p} ): Q \subset Q_0 \ \textrm{dyadic}, \ t \in (0,l^{p}) \}.\]
Here dyadic means dyadic with respect to $Q_0$, and hence the collection $\mathcal{B}$ consists of the lower parts $\widehat{R}$ of spatially dyadic short parabolic rectangles interpreted as metric balls with respect to
\[d((x,t),(x',t')) = \max \{ \abs{x-x'}_\infty, C_{ p} \abs{t-t'}^{1/p}  \}. \]
Notice that $(n+1)$-dimensional Lebesgue measure is doubling with respect to $d$.
 
We define a noncentered maximal function with respect to $\mathcal{B}$ as
\[M_{\mathcal{B}} f(x) = \sup_{\{x\} \subset B \in \mathcal{B}} \dashint_{B} f, \]
where the supremum is taken over all sets in $\mathcal{B}$ that contain $x$. 
By the Lebesgue differentiation theorem, we have 
\[\widehat{R}_0  \cap \{w > \lambda\} \subset \{M_{\mathcal{B}} (1_{\widehat{R}_0}w) > \lambda\} =: E\] 
up to a null set. Next we will construct a Calder\'on-Zygmund type cover. The idea is to use dyadic structure to deal with spatial coordinates, then separate the scales, and finally conclude, with one-dimensional arguments, the assumptions of Lemma \ref{lemma:cover1}.  

Define the slice $E_t = E \cap (\mathbb{R}^{n} \times \{t\})$ for fixed $t$. Since $\lambda \geq w_{\widehat{R}_0}$, we may find a collection of maximal dyadic cubes $Q_i^{t} \times \{t\} \subset E_t$ such that for each $Q_i$ there is $B_i^{t} \in \mathcal{B}$ with 
\[B_i^{t} \cap (Q_0 \times \{t\}) = Q_i^{t} \quad \textrm{and} \quad \dashint_{B_i^{t}}w > \lambda.\]
Clearly $\{B_i^{t}\}_{i}$ is pairwise disjoint and covers $E_t$. Moreover, since $Q_i^{t}$ is maximal, the dyadic parent $\widehat{Q}_i^{t}$ of $Q_i^{t}$ satisfies 
\[ \dashint_{\widehat{Q}_i^{t} \times I} w  \leq \lambda\]
for all intervals $I \ni t$ with $\abs{I} = l(\widehat{Q}_{i}^{t})^{p}$ and especially for the ones with $\widehat{Q}_{i}^{t} \times I \supset B_i^{t} $. Hence
\begin{equation}
\label{eq:quasi1}
\lambda < \dashint_{B_i^{t}} w \lesssim \dashint_{\widehat{Q}_i^{t} \times I} w \leq \lambda. 
\end{equation}

We gather the collections corresponding to $t \in (\tau, \tau + l_0^{p})$ together, and separate the resulting collection to subcollections as follows:
\[\mathcal{Q} = \{B_i^{t} :i \in\mathbb{Z}, t \in (0,l^{p})  \} = \bigcup_{j\in \mathbb{Z}} \mathcal{Q}_j,\]
where $\mathcal{Q}_j = \{Q \times I \in \mathcal{Q}:  \abs{Q} = 2^{-jn} \abs{Q_0}  \}$.
Each $\mathcal{Q}_j$ can be partitioned into subcollections corresponding to different spatial dyadic cubes $\mathcal{Q}_j = \bigcup_{i} \mathcal{Q}_{ji}$. Here 
\[\mathcal{Q}_{ji} = \{Q \times I \in \mathcal{Q}_j: Q = Q_i^{t} , t \in (\tau, \tau + l^{p}) \}.\]
If needed, we may reindex the Calder\'on-Zygmund cubes canonically with $j$ and $i$ such that $j$ tells the dyadic generation and $i$ specifies the cube such that $Q_{ji}^{t} =  Q_{ji}^{t'}$. Then
\[\bigcup_{B \in \mathcal{Q}_{ij} } B \cap \bigcup _{B' \in \mathcal{Q}_{i'j} } B' = \varnothing\]
whenever $i \neq i'$. Thus we may identify $\mathcal{Q}_{ji}$ with a collection of intervals and extract a covering subcollection with an overlap bounded by $2$. Hence we get a covering subcollection of $\mathcal{Q}_j$ with an overlap bounded by $2$, and hence a countable covering subcollection of $\mathcal{Q}$ such that its restriction to any dyadic scale has an overlap bounded by $2$. Denote the final collection by $\mathcal{F}$. Its elements are interpreted as lower halves of parabolic rectangles, that is, there are parabolic rectangles $P $ with $P^{-} \in \mathcal{F}$.  

Collect the parabolic halves $P^{-} \in \mathcal{F}$ with maximal side length to the collection $\Sigma_1$. Recursively, if $\Sigma_k$ is chosen, collect $P^{-} \in \mathcal{F}$ with equal maximal size such that 
\[P^{+} \cap \bigcup_{Q^{-} \in \bigcup_{i = 1 }^{k} \Sigma_i} Q^{+} = \varnothing \]
to the collection $\Sigma_{k+1}$. The collections $\Sigma_k$ share no elements, and their internal overlap is bounded by $2$. Since each $A \in \Sigma_k$ has equal size, the bounded overlap is inherited by the collection
\[\Sigma_k^{+} := \{A^{+}: A^{-} \in \Sigma_k \}.\]  
Moreover, by construction, if $A^{+} \in \Sigma_i^{+}$ and $B^{+} \in \Sigma_j^{+}$ with $i \neq j$ then $A^{+} \cap B^{+} = \varnothing$. Hence 
\[\mathcal{F}' := \bigcup_{i} \Sigma_i\]
is a collection such that 
\[\sum_{P^{-} \in \mathcal{F}'} 1_{P^{+}} \leq 2. \] 

According to \eqref{eq:quasi1} and Lemma \ref{lemma:cover1}, we get
\begin{align*}
w(E) &\leq \sum_{B \in \mathcal{F}} w(B) \lesssim \sum_{B \in \mathcal{F}} \lambda \abs{B} \\
& \leq   \sum_{P^{-} \in \mathcal{F}'}\left(\lambda  \abs{P^{-}} +  \sum_{\substack{B \in \mathcal{F} \\ B^{+} \cap P^{+} \neq \varnothing \\ \abs{B} < \abs{P}}}  \lambda \abs{B} \right)     \lesssim \lambda \sum_{P^{-} \in \mathcal{F}'} \abs{P^{+}}.
\end{align*}
Then 
\begin{align*}
w(E) &\lesssim_{\gamma} \lambda \sum_{P^{-} \in \mathcal{F}'} \abs{\check{P}} \lesssim \sum_{P^{-} \in \mathcal{F}'}   \lambda \abs{\check{P} \cap \{w > \beta \lambda \}} \\
& \leq   \int_{\bigcup_{S^{-} \in \mathcal{F}'} \check{S} \cap \{w > \beta \lambda \}} \sum_{P^{-} \in \mathcal{F}'} 1_{P^{+}}  \lesssim \lambda \abs{\widetilde{R}_0 \cap \{w > \beta \lambda \}}.
\end{align*}
\end{proof}

The fact that the sets in the estimate given by the above lemma are not equal is reflected to the reverse H\"older inequality as a time lag. This phenomenon is unavoidable, and it was noticed already in the one-dimensional case, see for instance \cite{Martin1993}.

\begin{theorem}
\label{thm:RHI}
Let $w \in A_q^{+}(\gamma)$ with $\gamma  \in (0,1)$. Then there exist $\delta > 0$ and a constant $C$ independent of $R$ such that
\[\left(\dashint_{R^{-}(0)} w^{\delta + 1} \right)^{1/(1+\delta)} \leq C \dashint _{R^{+}(0)} w .\]
Furthermore, there exists $\epsilon > 0$ such that $w \in A_{q-\epsilon}^{+}(\gamma)$.
\end{theorem}

\begin{proof}
We will consider a truncated weight $w := \min\{w, m\}$ in order to make quantities bounded. At the end, the claim for general weights will follow by passing to the limit as $m\to\infty$. Without loss of generality, we may take $R^{-} = Q \times (0,l^{p})$. Define $\widehat{R}$ and $\check{R}$ as in the previous lemma (see \eqref{def:hattu} and \eqref{def:mato}). In addition, let $\widetilde{R}$ be the convex hull of $\widehat{R}\cup \check{R}$. 

Let $E = \{w > w_{R^{-}}\}$. By Lemma \ref{lemma:quasistein}
\begin{align*}
\int_{R^{-} \cap E} w^{\delta +1}  
&= \abs{R^{-} \cap E} w_{R^{-}}^{\delta +1} + \delta \int_{w_{R^{-}}}^{\infty} \lambda^{\delta - 1} w(\{R^{-} \cap \{w > \lambda\} \}) \dla   \\
&\leq \abs{R^{-} \cap E} w_{R^{-}}^{\delta +1} + C  \delta \int_{w_{R^{-}}}^{\infty} \lambda^{\delta - 1} \abs{\{R  \cap \{w > \beta\lambda\} \}} \dla   \\
&\leq  \abs{R^{-} \cap E} w_{R^{-}}^{\delta +1} + C  \delta \int_{\widetilde{R}  \cap E} w^{\delta +1} ,
\end{align*}
which implies that
\begin{align*}
\int_{R^{-} \cap E }w^{\delta +1} & \leq \frac{1}{1-\delta C} \left(   \abs{R^{-} \cap E} w_{R^{-}}^{\delta +1} + C  \delta \int_{\widetilde{R} \setminus (R^{-}  \cap E)} w^{\delta +1}\right) .
\end{align*} 
Consequently
\begin{align}
\label{eq:rhi1}
\int_{R^{-} }w^{\delta +1} & \leq \frac{2-\delta C}{1-\delta C}   \abs{R^{-}} w_{R^{-}}^{\delta +1} + \frac{C \delta}{1-\delta C}   \int_{\widetilde{R}\setminus R^{-}   } w^{\delta +1} \nonumber \\
& = C_0   \abs{R^{-}} w_{R^{-}}^{\delta +1} + C_1  \delta \int_{\widetilde{R}\setminus R^{-}   } w^{\delta +1}.
\end{align}

Then we choose $l_1^{p} = 2^{-1} l^{p}$. We can cover $Q$ by $M_{np}$ subcubes $\{Q_i^{1}\}_{i=1}^{M_{np}}$ with $l(Q_i^{1}) = l_1$. Their overlap is bounded by $M_{np}$, and so is the overlap of the rectangles
\[\{R_i^{1-}\} = Q_i \times (l^{p}, \frac{3}{2} l^{p}) \]
that cover $\widetilde{R} \setminus R^{-}$ and share the dimensions of the original $R^{-}$. Hence we are in position to iterate. The rectangles $R_{ij}^{(k+1)-}$ are obtained from $R_i^{k-}$ as $R_i^{1-}$ were obtained from $R^{-} =: R_i^{0-}$, $i = 1, \ldots , M_{np}$. Thus
\begin{align*}
\label{eq:rhi2}
&\int_{R^{-} }w^{\delta +1} 
\leq  C_0   \abs{R^{-}} w_{R^{-}}^{\delta +1} + C_1  \delta \sum_{i=1}^{M_{np}} \int_{R_i^{1}   } w^{\delta +1} \\
&\leq \sum_{j=0}^{N}\left( C_0^{j+1} (C_1 \delta )^{j } \sum_{i=1}^{M_{np}} \abs{R_{i}^{j-}} w_{R_{i}^{j-}}^{\delta +1} \right) + (C_1 \delta M_{np})^{N} \int_{\bigcup_{i=1}^{M_{np}} \widetilde{R}_{i}^{N} \setminus R_{i}^{N-} } w^{\delta + 1} \\
&= I + II.
\end{align*}

For the inner sum in the first term we have
\[
\sum_{i=1}^{M_{np}} \abs{R_{i}^{j-}} w_{R_{i}^{j-}}^{\delta +1} 
\leq \sum_{i=1}^{M_{np}} 2^{-j \delta n } l^{- \delta(n+p)} \left(\int_{R_{i}^{j-}} w \right)^{\delta +1} 
\leq  2^{-j \delta n} l^{n+p} M_{np}^{\delta + 1} w_{R}^{\delta + 1}.
\]
Thus  
\begin{align*}
I \leq \left( \dashint_{R }w \right)^{1+\delta}  C_0 M_{np}^{\delta + 1} l^{n+p} \sum_{j=0}^{N}(C_1C_0 \delta )^{j}  2^{-j \delta n} ,
\end{align*}
where the series converges as $N \to \infty$ if $\delta $ is small enough. On the other hand, if $w$ is bounded, it is clear that $II \to 0$ as $N \to \infty$. This proves the claim for bounded $w$, hence for truncations $\min \{w,m\}$, and the general case follows from the monotone convergence theorem as $m \to \infty$. The self improving property of $A_q^{+}(\gamma)$ follows from applying the reverse H\"older inequality coming from the $A_{q'}^{-}(\gamma)$ condition satisfied by $w^{1-q'}$ and using Proposition \ref{prop:properties}.
\end{proof}

\begin{remark}
\label{remark:rhiimproving}
An easy subdivision argument shows that the reverse H\"older inequality can be obtained for any pair $R, t + R$ where $t > 0$. Namely, we can divide $R$ to arbitrarily small, possibly overlapping, subrectangles. Then we may apply the estimate to them and sum up. This kind of procedure has been carried out explicitly in \cite{Berkovits2011}.
\end{remark}

Now we are ready to state the analogue of Muckenhoupt's theorem in its complete form. Once it is established, many results familiar from the classical Muckenhoupt theory follow immediately.

\begin{theorem}
\label{thm:strongchar}
Let $\gamma_{i} \in (0,1)$, $i = 1,2,3$. Then the following conditions are equivalent:
\begin{itemize}
\item[(i)] $w \in A_q^{+}(\gamma_1)$,
\item[(ii)] the operator $M^{\gamma_2+}$ is of weighted weak type $(q,q)$ with the weight $w$,
\item[(iii)] the operator $M^{\gamma_3+}$ is of weighted strong type $(q,q)$ with the weight $w$.
\end{itemize}
\end{theorem}

\begin{proof}
Equivalence of $A_q^{+}$ and weak type follows from Theorem \ref{thm:weakchar}. Theorem \ref{thm:RHI} gives $A_{q-\epsilon}^{+}$, so (iii) follows from Marcinkiewicz interpolation and the final implication (iii) $\Rightarrow$ (ii) is clear.
\end{proof}

\section{Factorization and $A_1^{+}$ weights}
In contrast with the classical case, it is not clear what is the correct definition of the parabolic Muckenhoupt class $A_1^{+}$. One option is to derive a $A_1^{+}$ condition from the weak type $(1,1)$ inequality for $M^{\gamma+}$, and get a condition that coincides with the formal limit of $A_q^{+}$ conditions.  We propose a slightly different approach and  consider the class arising from factorization of the parabolic Muckenhoupt weights and characterization of the parabolic $\BMO$.

\begin{definition}
Let $\gamma \in [0,1)$. A weight $w> 0$ is in $ A_1^{+}(\gamma)$ if for almost every $z \in \mathbb{R}^{n+1}$, we have
\begin{equation}
\label{eq:MA1condition1}
M^{\gamma-} w(z) \leq [w]_{ A_1^{+}(\gamma)}  w(z) .
\end{equation}
The class $ A_1^{-}(\gamma)$ is defined by reversing the direction of time.
\end{definition}

The following proposition shows that, in some cases, the $A_1^{+}$ condition implies the $A_1$ type condition equivalent to the weak $(1,1)$ inequality. Moreover, if $\gamma = 0$, then the two conditions are equivalent.

\begin{proposition}
\label{prop:MA1}
Let $w \in A_1^{+}(\gamma)$ with $\gamma < 2^{1-p}$.
\begin{enumerate}
\item[(i)] For every parabolic rectangle $R$ it holds that
\begin{equation}
\label{eq:MA1condition}
\dashint_{R^{-}(2^{p-1}\gamma) } w \lesssim_{\gamma, [w]_{A_1^{+}}} \inf_{z \in R^{+}(2^{p-1}\gamma)} w(z) .
\end{equation}
\item[(ii)] For all $q > 1$ we have that $w \in A_q^{+}$.
\end{enumerate}

\end{proposition}
\begin{proof}
Denote $\delta = 2^{p-1} \gamma$. Take a parabolic rectangle $R_0$. We see that every $z \in R_0^{+}(\delta)$ is a center of a parabolic rectangle with $R^{-} (z,\gamma) \supset R_0^{-}(\delta)$ such that 
\[\dashint_{R^{-}(\delta)}w \lesssim \dashint_{R^{-}(z,\gamma)}w \leq M^{\gamma-}w(z) \lesssim w(z), \]
where the last inequality used  \eqref{eq:MA1condition1}. This proves (i). The second statement (ii) follows from the fact that  \eqref{eq:MA1condition} is an increasing limit of $A_q^{+}(\gamma)$ conditions, see 
Proposition \ref{def:MAq}.
\end{proof}

Now we will state the main result of this section, that is, the factorization theorem for the parabolic Muckenhoupt weights corresponding to the classical results, for example, in  \cite{Jones1980} and \cite{CJR1983}.

\begin{theorem}
\label{thm:factorization}
Let $\delta \in (0,1)$ and $\gamma \in (0,\delta 2^{1-p})$. A weight $w \in A_q^{+}(\delta) $ if and only if $w = u v^{1-p}$, where $u \in A_1^{+}(\gamma)$ and $v \in A_1^{-}(\gamma)$.
\end{theorem}
\begin{proof}
Let $u \in A_1^{+}(\gamma)$, $v \in A_1^{-}(\gamma)$ and fix a parabolic rectangle $R$. By Proposition \ref{prop:MA1}, for all $x \in R^{+}(\delta)$, we have 
\begin{align*}
u(x)^{-1} \leq \sup_{x \in R^{+}(\delta)} u(x)^{-1} = \left( \inf_{x \in R^{+}(\delta)} u(x) \right)^{-1} \lesssim  \left( \dashint_{R^{-}(\delta)} u \right)^{-1},
\end{align*}
and, for all $y \in R^{-}(\delta)$, we have the corresponding inequality for $v$, that is,
\begin{align*}
v(y)^{-1} \leq \sup_{y \in R^{-}(\delta)} v(y)^{-1} = \left( \inf_{y \in R^{-}(\delta)} v(y) \right)^{-1} \lesssim  \left( \dashint_{R^{+}(\delta)} v \right)^{-1}.
\end{align*}
Hence
\begin{align*}
&\left(\dashint_{R^{-}(\delta)} uv^{1-q} \right) \left(\dashint_{R^{+}(\delta)} u^{1-q'} v \right)^{p-1} \\
&    \lesssim    \left(\dashint_{R^{-}(\delta)} u  \right)\left( \dashint_{R^{+}(\delta)} v \right)^{1-q}  \left( \dashint_{R^{+}(\delta)}   v \right)^{q-1}\left( \dashint_{R^{-}(\delta)} u \right)^{-1}  = C ,
\end{align*}
which proves that $uv^{1-q} \in A_q^{+}(\delta)$. The finite constant $C$ depends only on $\gamma, \delta, [u]_{A_1^{+}(\gamma)}$ and $[v]_{A_1^{-}(\gamma)}$. 

For the other direction, fix $q \geq 2$ and $w \in A_q^{+}$. Define an operator $T$ as
\[Tf = (w^{-1/q} M^{\gamma -} (f^{q-1} w^{1/q}))^{1/(q-1)}  +   w^{1/q} M^{\gamma +} (f  w^{-1/q}).  \]
By boundedness of  the operators 
\[
M^{\gamma+} : L^{q}(w) \to L^{q}(w)
\quad\text{and}\quad M^{\gamma -} : L^{q'}(w^{1-p'}) \to L^{q'}(w^{1-p'})
\]
we conclude that $T : L^{q} \to L^{q} $ is bounded. Let 
\[B(w) := \norm{T}_{L^{q} \to L^{q}} \eqsim_{[w]_{A_q^{+}}} 1. \]

Take $f_0 \in L^{q}$ with $\norm{f_0}_{L^{q}} = 1$. Let
\[\phi = \sum_{i=1}^{\infty} (2B(w))^{-i} T^{i}f_0 \]
where $T^{i}$ simply means the $i$th iterate of $T$. We define 
\[u = w^{1/q} \phi^{q-1} \quad \textrm{and} \quad v = w^{-1/q}\phi. \]
Clearly $w = uv^{1-q}$. We claim that $u \in  A_1^{+}$ and $v \in  A_1^{-}$. Since $q \geq 2$ the operator $T$ is sublinear, and we obtain
\begin{align*}
T(\phi) &\leq 2B(w) \sum_{i=1}^{\infty} (2B(w))^{-(i+1)} T^{i+1}(f_0) \\
& = 2B(w) \left( \phi - \frac{T(f_0)}{2B(w)} \right) 
 \leq 2B(w) \phi.  
\end{align*}
Noting that $\phi = (w^{-1/q}u)^{1/(q-1)} = w^{1/q} v$ and inserting the above inequality into the definition of $T$, we obtain
\begin{align*}
M^{\gamma-} u \leq (2B(w))^{q-1} u \quad \textrm{and} \quad M^{\gamma +} v \leq 2B(w) v.
\end{align*}
This implies that $u \in A_1^{+}$ and $v \in A_1^{-}$ so the proof is complete for $q \geq 2$. Once the claim is known for $q \geq 2$, the complementary case $1 < q < 2$ follows from Proposition \ref{prop:properties} (ii).
\end{proof}

Next we will characterize $A_1^{+}$ weights as small powers of maximal functions up to a multiplication by bounded functions. The following result looks very much like the classical characterization of Muckenhoupt $A_1$ weights. However, we emphasize that even if the maximal operator $M^{\gamma+}$ is dominated by the Hardy-Littlewood maximal operator, the assumptions of the following lemma are not restrictive at all when it comes to the measure $\mu$. Indeed, the condition $M^{\gamma-} \mu < \infty$ almost everywhere still includes rather rough measures. For instance, their growth towards the positive time direction can be almost arbitrary, and the same property is carried over to the $A_1^{+}$ weights. 

\begin{lemma} 
\label{lemma:charofMA1}
\begin{enumerate} 
\item[(i)] Let $\mu $ be a locally finite nonnegative Borel measure on $\mathbb{R}^{n+1}$ such that $M^{-} \mu < \infty$ almost everywhere. If $\delta \in [0,1)$, then 
\[w:=(M^{-}\mu)^{\delta} \in A_1^{+}(0)\] 
with $[w]_{ A_1^{+}(0)}$ independent of $\mu$. 
\item[(ii)] Let $w \in A_1^{+}(\gamma')$. Then there exists a $\mu$ as above, $\delta \in [0,1)$ and $K$ with $K,K^{-1} \in L^{\infty}$ such that 
\[w = K (M^{\gamma -}\mu)^{\delta},\]
where $\gamma' < \gamma$.
\end{enumerate}
\end{lemma}
\begin{proof}
Let $x \in \mathbb{R}^{n+1}$ and fix a parabolic rectangle $R_0$ centered at $x$. Denote $\widetilde{B} = (2R_0)^{-}$. Decompose $\mu = \mu_1 + \mu_2$ where $\mu_1 = \mu \vert_{\widetilde{B} } $ and $\mu_2 = \mu \vert_{\widetilde{B}^{c}}$. Kolmogorov's inequality gives
\begin{align*}
\dashint_{R_0^{-}} (M^{-}\mu_{1} )^\delta &\leq C \abs{R_0^{-}}^{-\delta}\mu_{1}(\widetilde{B})^{\delta}   \leq C \left( \frac{\mu(\widetilde{B})}{ \lvert \widetilde{B} \rvert } \right)^{\delta}  \leq C M^{ -}\mu(x)^{\delta} .
\end{align*}

On the other hand, for any $y \in R_0^{-}$ and a rectangle $R(y,L) \cap (\widetilde{B})^{c} \neq \varnothing$,
we have $L \gtrsim l(R_0)$. Moreover, $R(y,L) \subset R(x,C L)$ so that
\[M^{-}\mu_2(y)  \lesssim M^{-}\mu(x) \]
and
\[
\dashint_{R_0^{-}} (M^{-}\mu)^{\delta} \leq  \dashint_{R_0^{-}}  (M^{ -}\mu_2)^{\delta} + \dashint_{R_0^{-}} (M^{-}\mu_{1}) ^\delta 
\lesssim M^{-}\mu(x)^{\delta}.
\]

To prove (ii), take $w \in A_1^{+}(\gamma')$ and a parabolic rectangle $R$ centered at $x$. By the reverse H\"older property (Theorem \ref{thm:RHI}), Remark \ref{remark:rhiimproving}, and inequality \eqref{eq:MA1condition1} we have
\begin{align*}
\left( \dashint_{R^{-}(\gamma)} w^{1+\epsilon} \right)^{1/(1+\epsilon)} \lesssim   w(x).
\end{align*}
Denote $\mu = w ^{1+\epsilon}$ and $\delta = 1/(1+\epsilon)$. By the Lebesgue differentiation theorem 
\[w(x) \leq M^{\gamma-}\mu(x)^{\delta} \lesssim w(x). \]
Hence 
\[K = \frac{w}{(M^{\gamma-}\mu)^{\delta}} \]
is bounded from above and from below, which proves the claim.
\end{proof}

\section{A characterization of the parabolic $\BMO$}
In this section we discuss the connection between parabolic Muckenhoupt weights and the parabolic $\BMO$. 
The parabolic $\BMO$ was explicitly defined by Fabes and Garofalo in \cite{FG1985}, who gave a simplified proof of the parabolic John-Nirenberg lemma in Moser's paper  \cite{Moser1964}. We consider a  slightly modified definition in order to make the parabolic $\BMO$ a larger space and a more robust class, see \cite{Saari2014}. Our definition has essentially the same connections to PDE as the one in \cite{FG1985}. Moreover, this extends the theory beyond the quadratic growth case and applies to the doubly nonlinear parabolic equations.

\begin{definition}
\label{def:parabmo}
A function $u \in L_{loc}^{1}(\mathbb{R}^{n+1})$ belongs to $\PBMO^{+}$, if there are constants $a_R$, that may depend on the parabolic rectangles $R$, such that
\begin{equation}
\label{eq:bmocondition}
\sup_{ R  } \left( \dashint_{  R ^{+}(\gamma)} (u-a_R)^{+}   +  \dashint_{ R ^{-}(\gamma)} (a_R- u)^{+}   \right) < \infty.
\end{equation}
for some $\gamma \in (0,1)$.
If \eqref{eq:bmocondition} holds with the time axis reversed, then $u \in \PBMO^{-}$.
\end{definition}

If  \eqref{eq:bmocondition} holds for some $\gamma \in (0,1)$, then it holds for all of them. Moreover, we can consider prolonged parabolic rectangles $Q \times (t-T l^{p},t+T l^{p})$ with $T > 0$ and still recover the same class of functions. These facts follow from the main result in \cite{Saari2014}, and they can be deduced from  results in \cite{Aimar1988} and in a special case from results in \cite{FG1985}. 

The fact that $\gamma >0$ is crucial. For example, the John-Nirenberg inequality (Lemma \ref{lemma:JN}) for the parabolic $\BMO$ cannot hold without a time lag. Hence a space with $\gamma = 0$ cannot be characterized through the John-Nirenberg inequality. The following lemma can be found in \cite{Saari2014}. See also \cite{FG1985} and \cite{Aimar1988}.

\begin{lemma}
\label{lemma:JN}
Let $u \in \PBMO^{+}$ and $\gamma \in (0,1)$. Then there are $A,B > 0$ depending only on $n,\gamma$ and $u$ such that
\begin{equation}
\label{JN+}
 \abs{  R ^{+}(\gamma) \cap \{ (u-a_R)^{+} > \lambda\}} \leq A e^{-B\lambda} \abs{R } 
\end{equation}
 and
\begin{equation}
 \label{JN-}
 \abs{ R ^{-}(\gamma) \cap \{ (a_R-u)^{+} > \lambda\}} \leq A e^{-B\lambda} \abs{R } .
\end{equation}
\end{lemma}

There are also more elementary properties that can be seen from the Definition \ref{def:parabmo}. Since we will need them later, they will be stated in the next proposition.

\begin{proposition}
\label{prop:parabmo}

\begin{enumerate}
\item[(i)] If $u,v \in \PBMO^{+}$ and $\alpha , \beta \in (0,\infty)$, then $\alpha u+ \beta v \in \PBMO^{+}$.
\item[(ii)] $u \in \PBMO^{+}$ if and only if $-u \in \PBMO^{-}$.
\end{enumerate}
\end{proposition}
\begin{proof}
For (i), note that 
\[(u+v - (a_{R}^{u} + a_{R}^{v}))^{+} \leq (u-a_R^{u})^{+} +(u-a_R^{v})^{+},\]
and an analogous estimate holds for the negative part. Hence $\alpha u + \beta v \in \PBMO^{+}  $ with 
\[a_R = \frac{a_{R}^{u} }{\alpha} + \frac{ a_{R}^{v}}{\beta}.\]

Since
\[(u-a_R)^{+} = ((-u) - (-a_R))^{-} \quad \textrm{and } \quad (u-a_R)^{-} = ((-u) - (-a_R))^{+} ,\]
the second assertion is clear.
\end{proof}

The goal of this section is to characterize the parabolic $\BMO$ in the sense of Coifman and Rochberg \cite{CR1980}. The Muckenhoupt theory developed so far gives a characterization for the parabolic Muckenhoupt weights, so what remains to do is to prove the equivalence of the parabolic $\BMO$ and the $A_q^{+}$ condition.

\begin{lemma}
\label{lemma:charbmo}
Let $q  \in (1,\infty)$ and $\gamma \in (0,1)$. Then
\begin{equation}
\label{eq:charofbmo}
\PBMO^{+} = \{-\lambda \log w : w \in A_q^{+}(\gamma), \lambda \in (0,\infty)  \}.
\end{equation}
\end{lemma}
\begin{proof}
We abbreviate $R^{\pm}(\gamma)  = R^{\pm}$ even if $\gamma  \neq 0$. For $u \in \PBMO^{+}$, Lemma \ref{lemma:JN} gives $\epsilon > 0$ such that
\begin{align*}
\dashint_{R^{-}} e^{ -\epsilon u}   = e^{-a_R\epsilon }\dashint_{R^{-}} e^{\epsilon( a_R - u)} \leq e^{-a_R\epsilon}\dashint_{R^{-}} e^{ \epsilon(a_R -  u)^{+} } \leq C_- e^{-a_R \epsilon}
\end{align*}
and, for some $q < \infty$,
\begin{align*}
\dashint_{R^{+}} e^{ \epsilon u/(q-1)}   &= e^{ a_R \epsilon / (q-1)}\dashint_{R^{+}} e^{  ( u - a_R)\epsilon / (q-1)} \\
& \leq e^{ a_R\epsilon / (q-1)}\dashint_{R^{+}} e^{ (  u - a_R)^{+} \epsilon / (q-1) } \leq  C_+ e^{ a_R\epsilon / (q-1)}
\end{align*}
so $w := e^{-  u \epsilon } \in A_q^{+}$ and $u = -\epsilon^{-1} \cdot \log w$ as it was claimed.

To prove the other direction, take $w \in A_q^{+}$ with $q \leq 2$. Choose 
\[a_R =  \log w_{R^{-}}.\] 
Then by Jensen's inequality and the parabolic Muckenhoupt condition, we have
\begin{align*}
\exp \dashint_{R^{+}} (a_R -  \log w)^{+}   &\leq \dashint_{R^{+}} \exp (a_R -  \log w)^{+} \\
&\leq 1 + \dashint_{R^{+}} \exp \left(a_R - \frac{1}{1-q'} \log w^{1-q'}\right)  \\
&\leq 1 + \exp(a_R) \left( \dashint_{R^{+}} w^{1-q'} \right)^{q-1} \\
& = 1 + w_{R^{-}} \left( \dashint_{R^{+}} w^{1-q'} \right)^{q-1} 
 \leq 1  + C_{A_q^{+}}.
\end{align*} 
On the other hand, again by Jensen's inequality,
\begin{align*}
\exp \dashint_{R^{-}}(  \log w - a_R)^{+}   &\leq \dashint_{R^{-}} \exp ( \log w- a_R)^{+} \\
&\leq 1 + \dashint_{R^{-}} \exp ( \log w - a_R)   \\
& \leq 1 + \exp(- a_R) \dashint_{R^{-}} w  \\
& \leq 1 + w_{R^{-}}^{- 1}w_{R^{-}}  \leq 2.
\end{align*} 
This implies that
\begin{align*}
&\log(2(1+C_{A_q^{+}})) \\
&\geq \dashint_{R^{+}}(-  \log w-(-a_R))^{+}   + \dashint_{R^{-}}(-a_R - (-\log w))^{+}  ,
\end{align*}
and $u = -\log w \in \PBMO^{+}$. Applying the same argument for $A_{q'}^{-}$ with $q > 2$ shows that $- \log w^{1-q'} \in \PBMO^{-}$ and consequently Proposition \ref{prop:parabmo} implies that $- (q'-1) \log w \in \PBMO^{+}$.
\end{proof}

The following Coifman-Rochberg \cite{CR1980} type characterization for the parabolic $\BMO$ is the main result of this section. Observe, that it gives us a method to construct functions of parabolic bounded mean oscillation with prescribed singularities.
 
\begin{theorem}
\label{thm:coif-roch}
If $f\in\PBMO^{+}$ then there exist $\gamma \in (0,1)$, constants $\alpha,\beta >0$, a bounded function $b \in L^{\infty}$ and non-negative Borel measures $\mu$ and $ \nu$ such that
\[f = - \alpha \log M^{\gamma- }\mu + \beta \log M^{\gamma+} \nu + b. \]
Conversely, any $f$ of the form above with $\gamma = 0$ and $M^{- }\mu,M^{+} \nu < \infty$ belongs to $\PBMO^{+}$.
\end{theorem}
\begin{proof}
Take first $f \in \PBMO^{+}$. By Lemma \ref{lemma:charbmo}
\[f = - C \log w\]
with $C > 0$ and $w \in  A_2^{+}$. By Theorem \ref{thm:factorization}, there are $u \in  A_1^{+}$ and $v \in  A_1^{-}$ satisfying the corresponding maximal function estimates \eqref{eq:MA1condition1} such that 
\[w = uv^{-1}.\]
By Lemma \ref{lemma:charofMA1}, there exist functions $K_u,K_v, K_u^{-1},K_v^{-1} \in L^{\infty}$ and non-negative Borel measures $\mu$ and $\nu$ such that
\[u = K_u (M^{\gamma-}\mu)^{\alpha}  \quad \textrm{and} \quad  v = K_v (M^{\gamma+}\nu)^{\beta}.\]
Hence $f$ is of the desired form. The other direction follows from Lemma \ref{lemma:charofMA1}.
\end{proof}

\section{Doubly nonlinear equation}

We begin with pointing out that the theory discussed here applies not only to  \eqref{intro:equation}
but also to the PDEs
\[
\frac{\partial(|u|^{p-2}u)}{\partial t} - \dive A(x,t,u,Du)  =0,  \quad1<p<\infty,
 \]
where $A$ satisfies the growth conditions
\[
A(x,t,u,Du) \cdot Du \geq C_0 \abs{Du}^{p} 
\]
and
\[
\abs{A(x,t,u,Du)} \leq C_1 \abs{Du}^{p-1}.
\]
See \cite{KK2007} and \cite{Saari2014} for more. 
For simplicity, we have chosen to focus on the prototype equation \eqref{intro:equation} here.  

\subsection{Supersolutions are weights}
We say that 
\[
v \in L^{p}_{loc}((-\infty,\infty);W_{\loc}^{1,p}(\R^{n+1}))
\] 
is a supersolution to \eqref{intro:equation} provided 
\[\int \left( |\nabla v|^{p-2} \nabla v \cdot \nabla \phi - |v|^{p-2} v \frac{\partial \phi}{\partial t} \right)  \geq 0 \]
for all non-negative $\phi \in C_0^{\infty}(\R^{n+1})$. If the reversed inequality is satisfied, we call $u$ a subsolution. If a function is both sub- and supersolution, it is a weak solution.

The definition above allows us to use the following a priori estimate, which is Lemma 6.1 in \cite{KK2007}. 
Similar results can also be found in \cite{Moser1964} and \cite{Trudinger1968}, but we emphasize that the following lemma applies to the full range $1<p<\infty$ instead of just $p=2$.

\begin{lemma}[Kinnunen--Kuusi \cite{KK2007}]
\label{lemma:kinnunenkuusi}
Suppose that $v>0$ is a supersolution of the doubly nonlinear equation in $\sigma R$ where $\sigma > 1$ and $R$ is a parabolic rectangle. Then there are constants $C = C(p,\sigma,n)$, $C'=C'(p,\sigma ,n)$ and $\beta = \beta(R)$ such that
\[ \abs{ R^{-}  \cap \{ \log v > \lambda + \beta + C' \} } \leq \frac{C}{\lambda^{p-1}} \abs{R^{-} }  \]
and
\[ \abs{ R^{+}  \cap \{ \log v < -\lambda + \beta - C' \} } \leq \frac{C}{\lambda^{p-1}} \abs{R^{+} }\]
for all $\lambda > 0$.
\end{lemma}
\begin{remark}
There is a technical assumption $v > \rho > 0$ in  \cite{KK2007}. 
However, this assumption can be removed, see \cite{IMM2014}. Indeed, Lemma 2.3 of \cite{IMM2014} improves the inequality (3.1) of \cite{KK2007} as to make the proof of the above lemma work with general $v > 0$ in the case of \eqref{intro:equation} or more general parabolic quasiminimizers.
\end{remark}

Let $v$ be a positive supersolution and set  $u = -\log v$. We apply Lemma \ref{lemma:kinnunenkuusi} together with Cavalieri's principle to obtain
\[  \dashint_{  R^{+}} (   u  -a_R)_{+}^{b}   +  \dashint_{  R^{-}} (a_R- u)_{+}^{b}  < C(p,\sigma,\gamma,n) \]
with $b =  \min\{(p-1)/2,1\}$. A general form of the John-Nirenberg inequality from \cite{Aimar1988} together with its local-to-global properties from \cite{Saari2014} can be used to obtain
\[  \dashint_{  R^{+}(\gamma)} (u-a_R)_{+}    +  \dashint_{  R^{-}(\gamma)} (a_R- u)_{+}  < C(p,\sigma,\gamma,n). \]
Hence $u = - \log v$ belongs to $\PBMO^{+}$ in the sense of Definition \ref{def:parabmo}. The computations required in this passage are carried out in detail in Lemma 6.3 of \cite{Saari2014}.
We collect the results into the following proposition, whose content, up to notation, is folklore by now.

\begin{proposition}\label{prop:logbmo}
Let $v > 0$ be a supersolution to \eqref{intro:equation} in $\R^{n+1}$. Then
\[
 u = -\log v \in \PBMO^{+}
\]
In addition, $v \in \cap_{q > 1} A_q^{+}$.
\end{proposition}

\begin{remark}
This gives a way to construct nontrivial examples of the parabolic Muckenhoupt weights and parabolic $\BMO$ functions.
\end{remark}

Since $\log v \in \PBMO^{-}$, we have that some power of the positive supersolution $w$ satisfies a local $A_2^{+}(\gamma)$ condition. This follows from Lemma \ref{lemma:charbmo}. However, working a bit more with the PDE, it is possible to prove a weak Harnack estimate which implies the improved weight condition stated in the above proposition. This has been done in \cite{KK2007}, but the refinement provided in \cite{IMM2014} is again needed in order to cover all positive supersolutions. 



\subsection{Applications}
The previous proposition asserts that the definitions of parabolic weights and parabolic BMO are correct from the point of view of doubly nonlinear equation. These properties can be used to deduce two interesting results, the second one of which is new. 
The first one is a global integrability result for supersolutions, see Theorem 6.5 from \cite{Saari2014}.
The second application of the parabolic theory of weights is related to singularities of supersolutions. It follows from Proposition \ref{prop:logbmo} and Theorem \ref{thm:coif-roch}. 
In qualitative terms, the following theorem tells quite explicitly what kind of functions the generic positive supersolutions are. 
\begin{theorem}
Let $v> 0$ be a supersolution to  \eqref{intro:equation} in $\mathbb{R}^{n+1}$. Then there are positive Borel measures $\nu$ and $\mu$ with
\[M^{\gamma -} \nu < \infty \quad \textrm{and} \quad M^{\gamma +} \mu < \infty , \]
numbers $\alpha , \beta > 0$, and a positive function $b$ with $b, b^{-1} \in L^{\infty}(\mathbb{R}^{n+1})$ so that 
\[v =  b \frac{(M^{\gamma -} \nu)^{\alpha}}{(M^{\gamma +} \mu)^{\beta}} .\]
\end{theorem} 

\end{document}